\documentclass[a4paper,12pt]{amsart}

\usepackage[all]{xy}
\usepackage{amsfonts}
\usepackage[mathcal]{eucal}
\usepackage{eufrak}
\usepackage{amssymb}
\usepackage{amsmath}
\usepackage{mathrsfs} 

\usepackage[pagebackref,colorlinks]{hyperref}

\newcommand{\lmps}{\longmapsto}

\newcommand{\ra}{\rightarrow}

\newcommand{\mi}{\setminus}
\newcommand{\lra}{\longrightarrow}

\newcommand{\mar}{\mathcal{R}}

\newcommand{\conggr}{\cong_{\mathrm{gr}}}
\newcommand{\gr}{\mathrm{gr}}
\newcommand{\Pgr}{\mathcal P \mathrm{gr}}
\newcommand{\hygrmod}{\textrm{-}\mathfrak{gr}\textrm{-}\mathcal{M}\mathfrak{od}}

\mathchardef\mhyphen="2D

\newcommand{\va}{\varphi}
\newcommand{\Ga}{\Gamma}
\newcommand{\ga}{\gamma}
\newcommand{\al}{\alpha}

\newcommand{\de}{\delta}
\newcommand{\la}{\lambda}

\newcommand\CK[1][1]{\operatorname{CK}_{#1}}

\newcommand\ZK[1][1]{\operatorname{ZK}_{#1}} 
\newcommand\SK[1][1]{\operatorname{SK}_{#1}} 

\newcommand{\End}{\operatorname{End}}

\newcommand{\Hom}{\operatorname{Hom}}

\newcommand{\id}{\operatorname{id}}

\newcommand{\op}{\operatorname{op}}
\newcommand{\HOM}{\operatorname{HOM}}
\newcommand{\GL}{\operatorname{GL}}

\theoremstyle{plain}
\newtheorem{thm}{Theorem}[section]

\newtheorem{cor}[thm]{Corollary}
\newtheorem{prop}[thm]{Proposition}

\theoremstyle{remark}
\newtheorem{remark}[thm]{Remark}

\newtheorem{example}[thm]{Example}

\theoremstyle{definition}

\newtheorem{defin}[thm]{Definition}

\title{On Graded Simple Algebras}

\begin{thanks}
{The first author acknowledges the support of EPSRC first grant
scheme EP/D03695X/1.}
\end{thanks}

\author[Roozbeh Hazrat]{Roozbeh Hazrat}

\author[Judith R. Millar]{Judith R. Millar}

\address{
Department of Pure Mathematics\\
Queen's University\\
Belfast BT7 1NN\\
U.K.}

\email[Roozbeh Hazrat]{r.hazrat@qub.ac.uk}

\email[Judith R. Millar]{jmillar12@qub.ac.uk}


\begin{document}

\begin{abstract}

This note begins by observing that a graded central simple algebra,
graded by an abelian group, is a graded Azumaya algebra and it is
free over its centre. For a graded Azumaya algebra $A$ free over its
centre $R$, we show that $K_i^{\gr} (A)$ is ``very close'' to
$K_i^{\gr}(R)$, where $K_i^{\gr} (R)$ is defined to be $K_i( \Pgr
(R))$. Here $\Pgr (R)$ is the category of graded finitely generated
projective $R$-modules and $K_i, \,i\geq 0$, are the Quillen
$K$-groups.
\end{abstract}

\maketitle

\section{Introduction}

Let $R$ be a commutative ring and $A$ be an algebra over $R$ which
is finitely generated as an $R$-module. If for any maximal ideal $m$
of $R$, the algebra $A\otimes_R R/m$ is a central simple
$R/m$-algebra, then $A$ is called an Azumaya algebra. This is
equivalent to saying that $A$ is a faithfully projective $R$-module,
and the natural $R$-algebra homomorphism $A\otimes_R A^{\op}  \ra
\End_R(A)$ is an isomorphism
(see~\cite[Thm.~\small{III}.5.1.1]{knus}). In \cite{hm} it was
proven that for an Azumaya algebra $A$ free over its centre $R$ of
rank $n$, the Quillen $K$-groups of $A$ are isomorphic to the
$K$-groups of its centre up to $n$-torsion, i.e.,
\begin{equation} \label{propp}
K_i(A) \otimes \mathbb Z[1/n] \cong K_i(R) \otimes \mathbb Z[1/n].
\end{equation}
Boulagouaz \cite[Prop.~5.1]{boulag} and Hwang and Wadsworth
\cite[Cor.~1.2]{hwcor} observed that a graded central simple
algebra, graded by a torsion free abelian group, is an Azumaya
algebra; thus its $K$-theory can be estimated by the above result.


This note studies graded central simple algebras graded by an
arbitrary abelian group. We observe that a graded central simple
algebra, graded by an abelian group, is a (graded) Azumaya algebra
(Theorem~\ref{gcsaazumayaalgebra}), which extends the result of
\cite{boulag,hwcor} to graded rings in which the grade group is not
totally ordered. Thus its $K$-theory can also be estimated by
(\ref{propp}). We then study the graded $K$-theory of graded Azumaya
algebras. We introduce an abstract functor called a graded $\mathcal
D$-functor defined on the category of graded Azumaya algebras over a
commutative graded ring $R$ (Definition~\ref{gradeddfunctor}), and
show that the range of this functor is the category of bounded
torsion abelian groups (Theorem~\ref{grtorsionthm}). We then prove
that the kernel and cokernel of the $K$-groups are graded $\mathcal
D$-functors, which allows us to show that, for a graded Azumaya
algebra $A$ free over $R$, we have a relation similar to
\eqref{propp} in the graded setting (see
Theorem~\ref{grazumayafreethm}).

This note is organised as follows. We begin
Section~\ref{sectiongcsa} by recalling some definitions, many of
which can be found in \cite{hwcor, grrings}, though not always in
the generality that we require. We then study graded central simple
algebras graded by an arbitrary abelian group and observe that they
are Azumaya algebras (Theorem~\ref{gcsaazumayaalgebra}). In order to
do so, we need to rewrite the standard results from the literature
in the setting of arbitrary graded rings. We observe that the tensor
product of two graded central simple $R$-algebras is graded central
simple (Propositions~\ref{tensorgradedsimple} and
\ref{tensorgradedcentral}). This result has been proven by Wall  for
$\mathbb Z / 2 \mathbb Z$ -graded central simple algebras (see
\cite[Thm.~2]{wall}), and by Hwang and Wadsworth for $R$-algebras
with a totally ordered, and hence torsion-free, grade group (see
\cite[Prop.~1.1]{hwcor}).

In Section~\ref{sectiongrdfunctors} we study the graded $K$-theory
of graded Azumaya algebras by introducing an abstract functor called
a graded $\mathcal D$-functor, which is defined on the category of
graded Azumaya algebras over a commutative graded ring $R$
(Definition~\ref{gradeddfunctor}). Similar concepts have been
studied in \cite{sk12001, azumayask1, hm}, where functors have been
defined on the category of central simple algebras and the category
of Azumaya algebras.


\section{Graded Central Simple Algebras}\label{sectiongcsa}

We begin this section by recalling some basic definitions in the
graded setting. A unital ring $R = \bigoplus_{ \ga \in \Ga} R_{\ga}$
is called a \emph{graded ring} if $\Ga$ is a group, each $R_{\ga}$
is a subgroup of $(R, +)$ and $R_{\ga} \cdot R_{\delta} \subseteq
R_{\ga + \delta}$ for all $\ga, \delta \in \Ga$. We remark that
although $\Ga$ is initially an arbitrary group which is not
necessarily abelian, we will write $\Ga$ as an additive group.
The elements of $R_\ga$ are called \emph{homogeneous of degree
$\ga$} and we write deg$(x) = \ga$ if $x \in R_{\ga}$. We set
\begin{flalign*}
&\Ga_{R} = \big \{ \ga \in \Ga : R_{\ga} \neq \{0 \} \big \}, \;
\textrm{ the support (or grade set) of $R$, }\\
&\Ga_{R}^* = \big \{ \ga \in \Ga : R_{\ga}^* \neq \emptyset \big \},
\textrm{ the support of invertible homogenous elements of $R$}\\
\textrm{ and \; }& R^{h} = \bigcup_{\ga \in \Ga_{R}} R_{\ga},
\;\;\;\;\; \textrm{ the set of homogeneous elements of
$R$.}\end{flalign*} Here $R^*$ is the set of invertible elements of $R$.
Note that  the support of $R$ is not necessarily
a group, and that $1_R$ is homogeneous of degree zero. An ideal $I$
of $R$ is called a \emph{graded ideal} if
$$
I= \bigoplus_{ \ga \in \Gamma} (I \cap R_{\ga}).
$$
Let $S= \bigoplus_{ \ga \in \Ga'} S_{\ga}$ be another graded ring
and suppose there is a group $\Delta$ containing $\Ga$ and $\Ga '$
as subgroups. The graded ring $R$ can be written as $R = \bigoplus_{
\ga \in \Delta} R_{\ga}$ with $R_\ga = 0$ if $\ga \in \Delta \mi
\Ga_R$, and similarly for $S$. Then a \emph{graded ring
homomorphism} $f:R \ra S$ is a ring homomorphism such that
$f(R_{\ga}) \subseteq S_{\gamma}$ for all $\ga \in \Delta$. If $f$
is bijective, then $f$ is a \emph{graded isomorphism}. A graded ring
$R$ is said to be {\em graded simple} if the only graded two-sided
ideals of $R$ are $\{ 0 \}$ and $R$. A graded ring $D = \bigoplus_{
\ga \in \Ga} D_{\ga}$ is called a \emph{graded division ring} if
every non-zero homogeneous element has a multiplicative inverse,
where it follows easily that $\Ga_D$ is a group.

We say that a group $(\Ga , +)$  {\em acts freely} (as a left
action) on a set $\Ga'$ if for all $\ga$, $\ga' \in \Ga$, $\de \in
\Ga'$, we have $\ga + \de = \ga' + \de$ implies $\ga = \ga'$, where
$\ga + \de$ denotes the image of $\de$ under the action of $\ga$. A
\emph{graded left $R$-module} $M$ is defined to be an $R$-module
with a direct sum decomposition $M=\bigoplus_{\ga \in \Ga'}
M_{\ga}$, where $M_{\ga}$ are abelian groups and $\Ga$ acts freely
on the set $\Ga'$, such that $R_{\ga} \cdot M_{\la} \subseteq M_{\ga
+ \la}$ for all $\ga \in \Ga_R, \la \in \Ga'$. {\it From now on,
unless otherwise stated, a graded module will mean a graded left
module.} A graded $R$-module $M$ is said to be {\em graded simple}
if the only graded submodules of $M$ are $\{ 0 \}$ and $M$, where
graded submodules are defined in the same way as graded ideals. A
\emph{graded free $R$-module} $M$ is defined to be a graded
$R$-module which is free as an $R$-module with a homogeneous base.

Let $N= \bigoplus_{\ga \in \Ga''} N_{\ga}$ be another graded
$R$-module, such that there is a group $\Delta$ containing $\Ga'$
and $\Ga''$ as subgroups, where $\Ga$ acts freely on $\Delta$. A
\emph{graded $R$-module homomorphism} $f:M \ra N$ is an $R$-module
homomorphism such that $f(M_{\de}) \subseteq N_{\de}$ for all $\de
\in \Delta$. Let $\Hom_{R \hygrmod}(M,N)$ denote the group of graded
$R$-module homomorphisms, which is an additive subgroup of $\Hom_R
(M,N)$. A graded $R$-module homomorphism may also shift the grading
on $N$. For each $\de \in \Delta$, we have a subgroup of
$\mathrm{Hom}_R(M,N)$ of $\de$-shifted homomorphisms
$$ {\mathrm{Hom}_R(M,N)}_{\de} = \{ f \in \mathrm{Hom}_R(M,N) :
f(M_{\ga}) \subseteq N_{\ga + \de} \textrm{ for all } \ga \in \Delta
\}. $$  Let $\mathrm{HOM}_R(M,N) = \bigoplus_{\de \in \Ga}
\mathrm{Hom}_R(M,N)_{\de}.$ For some $\de \in \Delta$, we define the
$\de$-shifted $R$-module $M(\de)$ as $M(\de) =\bigoplus_{\ga \in
\Delta} M(\de)_\ga$ where $M(\de)_\ga = M_{\ga+\de}$. Then
$$\Hom_R(M,N)_\de = \Hom_{R \hygrmod}(M,N(\de)) =
\Hom_{R \hygrmod}(M(-\de),N).$$ If $M$ is finitely generated, then
$\HOM_R(M,N) = \Hom_R(M,N)$ (see \cite[Cor.~2.4.4]{grrings}). Note
that $R(\de)\cong_{\gr} R$ as graded $R$-modules if and only if $\de
\in \Gamma_R^*$.

In the following Proposition we are considering graded modules over
graded division rings. We note that the grade groups here are
defined as above; that is, we do not initially assume them to be
abelian or torsion-free.

\begin{prop}\label{gradedfree}
Let $\Ga$ be a group which acts freely on a set $\Ga'$. Let $R =
\bigoplus_{ \ga \in \Ga} R_{\ga}$ be a graded division ring and
$M=\bigoplus_{\ga \in \Ga'} M_{\ga}$ be a graded module over $R$.
Then $M$ is a graded free $R$-module. More generally, any linearly
independent subset of $M$ consisting of homogeneous elements can be
extended to form a homogeneous basis of $M$. Furthermore, any two
homogenous bases have the same cardinality and if $N$ is
a graded submodule of $M$, then
\begin{equation}\label{dimensioneqn}
\dim_R (N) + \dim_R(M/N) = \dim_R (M).
\end{equation}
\end{prop}
\begin{proof}
The proof follows the standard proof in the non-graded setting (see for example
\cite[Thms.~\small{IV}.2.4,~2.7,~2.13]{hungerford}), or the graded
setting (see \cite[Thm.~3]{boulaggradedandtame},
\cite[p.~79]{hwcor}, \cite[Prop.~4.6.1]{grrings}); however extra
care needs to be given since the grading is neither abelian nor
torsion free.
\end{proof}

A \emph{graded field} $R = \bigoplus_{ \ga \in \Ga} R_{\ga}$ is
defined to be a commutative graded division ring. Note that the
support of a graded field is an abelian group. Let $\Ga'$ be another
group such that there is a group $\Delta$ containing $\Ga$ and
$\Ga'$ as subgroups. A \emph{graded $R$-algebra} $A= \bigoplus_{\ga
\in \Ga'} A_{\ga}$ is a graded ring which is an $R$-algebra such
that the associated ring homomorphism $\varphi :R \ra Z(A)$ is a
graded homomorphism. A graded algebra $A$ over $R$ is said to be a
{\it graded central simple algebra} over $R$ if $A$ is a graded
simple ring, $Z(A) =R$, and $[A:R] < \infty$. Note that since the
centre of $A$ is a graded field, by Proposition~\ref{gradedfree},
$A$ is graded free over its centre, so the dimension of $A$ over $R$
is uniquely defined.

Let $A= \bigoplus_{\ga \in \Ga'} A_{\ga}$ and $B= \bigoplus_{\ga \in
\Ga''} B_{\ga}$ be graded $R$-algebras, such that there is a group
$\Delta$ containing $\Ga'$ and $\Ga''$ as subgroups with $\Ga'
\subseteq Z_{\Delta} (\Ga'')$, where $Z_{\Delta} (\Ga'')$ is the set
of elements of $\Delta$ which commute with $\Ga''$.  Then $A
\otimes_R B$ has a natural grading as a graded $R$-algebra given by
$A \otimes_R B = \bigoplus_{\ga \in \Delta} (A \otimes_R B)_{\ga}$
where:
$$
(A \otimes_R B)_{\ga} = \left\{ \sum_i a_i \otimes b_i : a_i \in
A^h, b_i \in B^h, \deg(a_i)+\deg(b_i) = \ga \right\}
$$
Note that the condition $\Ga' \subseteq Z_{\Delta} (\Ga'')$ is
needed to ensure that the multiplication on $A \otimes_R B$ is well
defined. Moreover, for the following Proposition, we require that
the group $\Delta$ is an abelian group.

For a graded ring $A= \bigoplus_{\ga \in \Ga'} A_{\ga}$, let
$A^{\op}$ denote the opposite graded ring, where the grade group of
$A^{\op}$ is the opposite group $\Ga'^{\op}$. So, for a graded
$R$-algebra $A$, in order to define $A \otimes_R A^{\op}$, we note
that the grade group of $A$ must be abelian. Thus we will now assume
that for a graded $R$-algebra $A= \bigoplus_{\ga \in \Ga'} A_{\ga}$,
the group $\Ga'$ is in fact an abelian group.

By combining Propositions \ref{tensorgradedsimple} and
\ref{tensorgradedcentral}, we show that the tensor product of two
graded central simple $R$-algebras is graded central simple, where
the grade groups $\Ga'$ and $\Ga''$, as below, are abelian but not
necessarily torsion-free. This has been proven by Wall for graded
central simple algebras with $\mathbb Z / 2 \mathbb Z$ as the
support (see \cite[Thm.~2]{wall}), and by Hwang and Wadsworth for
$R$-algebras with a torsion-free grade group (see
\cite[Prop.~1.1]{hwcor}).

\begin{prop}\label{tensorgradedsimple}

Let $\Ga$, $\Ga'$ and $\Ga''$ be abelian groups such that there is
an abelian group $\Delta$ containing $\Ga$, $\Ga'$ and $\Ga''$ as
subgroups. Let $R  = \bigoplus_{ \ga \in \Ga} R_{\ga}$ be a graded
field and let $A= \bigoplus_{\ga \in \Ga'} A_{\ga}$ and $B=
\bigoplus_{\ga \in \Ga''} B_{\ga}$ be graded $R$-algebras. If $A$ is
graded central simple over $R$ and $B$ is graded simple, then $A
\otimes_R B$ is graded simple.
\end{prop}

\begin{proof}
Let $I$ be a graded two-sided ideal of $A \otimes B$, with $I \neq
0$. We will show that $A \otimes B = I$. First suppose $a \otimes b$
is a homogeneous element of $I$, where $a \in A^h$ and $b \in B^h$.
Then $A$ is the graded two-sided ideal generated by $a$, so there
exist $a_i ,$ $a'_i \in A^h$ with $1= \sum a_i a a'_i$. Then
$$\sum (a_i \otimes 1)( a \otimes b)( a'_i \otimes 1) = 1 \otimes b$$ is an
element of $I$. Similarly, $B$ is the graded two-sided ideal
generated by $b$. Repeating the above argument shows that $1 \otimes
1$ is an element of $I$, proving $I= A\otimes B$ in this case.

Now suppose there is an element $x \in I^h$, where $x = a_1 \otimes
b_1 + \cdots + a_k \otimes b_k$, with $a_j \in A^h$, $b_j \in B^h$
and $k$ as small as possible. Note that since $x$ is homogeneous,
$\deg(a_j) +\deg (b_j) = \deg (x)$ for all $j$. By the above
argument we can suppose that $k
>1$. As above, since $a_k \in A^h$, there are $c_i ,$ $c'_i \in A^h$
with $1= \sum c_i a_k c'_i$. Then
$$
\sum (c_i \otimes 1)x( c'_i \otimes 1) = \left( \sum (c_i a_1
c'_i) \right) \otimes b_1 + \cdots + \left( \sum (c_i a_{k-1} c'_i)
\right) \otimes b_{k-1} + 1 \otimes b_k,
$$
where the terms $\left( \sum_i (c_i a_j c'_i)\right) \otimes b_j$
are homogeneous elements of $A \otimes B$. Thus, without loss of
generality, we can assume that $a_k = 1$. Then $a_k$ and $a_{k-1}$
are linearly independent, since if $a_{k-1} = \la a_k$ with $\la \in
R$, then $a_{k-1} \otimes b_{k-1} + a_k \otimes b_k = a_k \otimes
(\la b_{k-1} +b_k)$, which is homogeneous and thus gives a smaller
value of $k$.

Thus $a_{k-1} \notin R=Z(A)$, and so there is a homogeneous element
$a \in A$ with $a a_{k-1} - a_{k-1} a \neq 0$. Consider the
commutator $$(a \otimes 1) x - x (a \otimes 1) = (aa_1 - a_1 a )
\otimes b_1 + \cdots + (a a_{k-1} - a_{k-1} a)\otimes b_{k-1},$$
where the last summand is not zero. If the whole sum is not zero,
then we have constructed a homogeneous element in $I$ with a smaller
$k$. Otherwise suppose the whole sum is zero, and write $c = a
a_{k-1} - a_{k-1} a $. Then we can write $c \otimes b_{k-1} =
\sum_{j=1}^{k-2} x_j \otimes b_j$ where $x_j = -(a a_j -a_j a)$.
Since $0\not = c \in A^h$ and $A$ is the graded two-sided ideal
generated by $c$, using the same argument as above, we have
\begin{equation} \label{li}
1 \otimes b_{k-1} = x'_1 \otimes b_1 + \cdots + x'_{k-2} \otimes
b_{k-2}\end{equation} for some $x'_j \in A^h$. Since $b_1 , \ldots ,
b_{k-1}$ are linearly independent homogeneous elements of $B$, they
can be extended to form a homogeneous basis of $B$, say $\{b_i\}$,
by Proposition~\ref{gradedfree}. Then $\{ 1 \otimes b_i\}$ forms a
homogeneous basis of $A\otimes_R B$ as a graded $A$-module, so in
particular they are $A$-linearly independent, which is a
contradiction to equation \eqref{li}. This reduces the proof to the
first case.
\end{proof}

\begin{prop}\label{tensorgradedcentral}

Let $\Ga$, $\Ga'$ and $\Ga''$ be abelian groups such that there is
an abelian group $\Delta$ containing $\Ga$, $\Ga'$ and $\Ga''$ as
subgroups. Let $R$ be a graded field and let $A= \bigoplus_{\ga \in
\Ga'} A_{\ga}$ and $B= \bigoplus_{\ga \in \Ga''} B_{\ga}$ be graded
$R$-algebras. If $A' \subseteq A$ and $B' \subseteq B$ are graded
subalgebras, then
$$Z_{A \otimes B}(A' \otimes B') = Z_A (A') \otimes Z_B(B').$$
In particular, if $A$ and $B$ are central over $R$, then $A
\otimes_R B$ is central.
\end{prop}

\begin{proof}
First note that by Proposition~\ref{gradedfree}, $A',B',Z_A(A')$ and
$Z_B(B')$ are free over $R$, and thus one can consider $Z_{A \otimes
B}(A' \otimes B')$ and  $Z_A (A') \otimes Z_B(B')$ as subalgebras of
$A\otimes B$.

The inclusion $\supseteq$ follows immediately. For the reverse
inclusion, let $x \in Z_{A \otimes B}(A' \otimes B')$. Let $b_1 ,
\ldots , b_n$ be a homogeneous basis for $B$ over $R$ which exists
thanks to Proposition~\ref{gradedfree}. Then $x$ can be written
uniquely as $x = x_1 \otimes b_1 + \cdots + x_n \otimes b_n$ for
$x_i \in A$ (see \cite[Thm.~IV.5.11]{hungerford}). For every $a \in
A'$, $(a \otimes 1) x = x (a \otimes 1)$, so
$$ (a x_1) \otimes b_1 + \cdots + (a x_n) \otimes b_n   = (x_1 a)
\otimes b_1 + \cdots + (x_n a) \otimes b_n.$$ By the uniqueness of
this representation we have $x_i a = a x_i$, so that $x_i \in Z_A
(A')$ for each $i$.  Thus we have shown that $x \in Z_A(A')
\otimes_R B$. Similarly, let $c_1 , \ldots , c_k$ be a homogeneous
basis of $Z_A(A')$. Then we can write $x$ uniquely as $x = c_1
\otimes y_1 + \cdots + c_k \otimes y_k$ for $y_i \in B$. A similar
argument to above shows that $y_i \in Z_B(B')$, completing the
proof.
\end{proof}

\begin{thm}\label{gcsaazumayaalgebra}
Let $\Ga$ and $\Ga'$ be abelian groups such that there is an abelian
group $\Delta$ containing $\Ga$ and $\Ga'$ as subgroups. Let $A =
\bigoplus_{\ga \in \Ga'} A_{\ga}$ be a graded central simple algebra
over the graded field $R = \bigoplus_{\ga \in \Ga} R_{\ga}$. Then
$A$ is an Azumaya algebra over $R$.
\end{thm}

\begin{proof}
Since $A$ is graded free of finite rank, it follows that $A$ is
faithfully projective over $R$. There is a natural graded
$R$-algebra homomorphism $\psi: A \otimes_R A^{\op} \ra \End_R (A)$
defined by $\psi(a\otimes b)(x)=axb$ where $a,x \in A$, $b \in
A^{\op}$. By Proposition~\ref{tensorgradedsimple}, the domain is
graded simple, so $\psi$ is injective. Hence the map is surjective
by dimension count, using equation~(\ref{dimensioneqn}). This shows
that $A$ is an Azumaya algebra over $R$, as required.
\end{proof}

For a graded field $R$, this theorem shows that a graded central
simple $R$-algebra, graded by an abelian group $\Ga'$, is an Azumaya
algebra over $R$. One can not extend the theorem to non-abelian
grading. Consider a finite dimensional division algebra $D$  and a
group $G$ and consider the group ring $DG$. This is clearly a graded
simple algebra (in fact a graded division ring) and if $G$ is
abelian the above theorem implies that $DG$ is an Azumaya algebra.
However in general, for an arbitrary group $G$,  $DG$ is not always
an Azumaya algebra. In fact DeMeyer and Janusz \cite{demjan} have
shown the following: the group ring $RG$ is an Azumaya algebra if
and only if $R$ is Azumaya, $[G:Z(G)] < \infty$ and $[G,G]$, the
commutator subgroup of $G$, has finite order $m$ and $m$ is
invertible in $R$.

\begin{cor}\label{gcsacor}
Let $\Ga$ and $\Ga'$ be abelian groups such that there is an abelian
group $\Delta$ containing $\Ga$ and $\Ga'$ as subgroups. Let $A =
\bigoplus_{\ga \in \Ga'} A_{\ga}$ be a graded central simple algebra
over its graded centre $R = \bigoplus_{\ga \in \Ga} R_{\ga}$ of
degree $n$. Then for any $i \geq 0$,
$$K_i(A) \otimes \mathbb Z[1/n] \cong K_i(R) \otimes \mathbb
Z[1/n].$$
\end{cor}

\begin{proof}
By Theorem~\ref{gcsaazumayaalgebra}, a graded central simple algebra
$A$ over $R$ is an Azumaya algebra. From
Proposition~\ref{gradedfree}, since $R$ is a graded field, $A$ is a
free $R$-module. The corollary now follows immediately from
\cite[Thm.~6]{hm} (or see (\ref{propp})), since $A$ is an Azumaya
algebra free over its centre.
\end{proof}

\section{Graded $K$-theory of Azumaya algebras} \label{sectiongrdfunctors}

Corollary~\ref{gcsacor} above shows that the $K$-theory of a graded
division algebra is very close to the $K$-theory of its centre,
where this follows immediately from the corresponding result in the
non-graded setting (see~\cite[Thm.~6]{hm}). Note that for the
$K$-theory of a graded central simple algebra $A$, we are
considering $K_i(A)  = K_i (\mathcal P (A))$, where $\mathcal P (A)$
denotes the category of finitely generated projective $A$-modules.
But in the graded setting, there is also the category of graded
finitely generated projective modules over a given graded ring,
which is what we consider here. Below we define an abstract functor
called a graded $\mathcal D$-functor
(Definition~\ref{gradeddfunctor}), and show that its range is the
category of bounded torsion abelian groups. We use this to show that
a similar result to the above Corollary also holds when we consider
graded projective modules over a graded ring.

Let $\Ga$ be an abelian group and let $R= \bigoplus_{ \ga \in \Ga}
R_{\ga}$ be a commutative $\Ga$-graded ring. We will consider the
category $R \hygrmod$ which is defined as follows: the objects are
$\Ga$-graded left $R$-modules, and for two objects $M$, $N$ in
$R\hygrmod$, the morphisms are defined as
$$
\Hom_{R \hygrmod}(M,N) = \{ f \in \Hom_R (M,N) : f (M_\ga )
\subseteq N_\ga \textrm{ for all } \ga \in \Ga\}.
$$
\emph{Through out this section, unless otherwise stated, we will
assume that $\Ga$ is an abelian group, $R$ is a fixed commutative
$\Ga$-graded ring and all graded rings, graded modules and graded
algebras are also $\Ga$-graded.}

Let $A$ be a graded ring and let $(d) = (\de_1 , \ldots , \de_n)$,
where each $\de_i \in \Ga$. Then we have a graded ring $M_n
(A)(d)$,\label{pagerefforgrmatrixrings} where $M_n (A)(d)$ means the
$n \times n$-matrices over $A$ with the degree of the $ij$-entry
shifted by $\delta_i - \delta_j$. Thus, the $\varepsilon$-component
of $M_n(A)(d)$ consists of matrices with the $ij$-entry in
$A_{\varepsilon + \delta_i - \delta_j}$. Consider
$$A^n(d) = \bigoplus_{\ga \in \Ga} (A(\de_1)_{\ga} \oplus \cdots
\oplus A(\de_n)_{\ga}) $$ where $A(\de_i)_{\ga}$ is the
$\ga$-component of the $\de_i$-shifted graded $A$-module $A(\de_i)$.
Note that for each $i$, $1 \leq i \leq n$, the basis element $e_i$
of $A^n(d)$ is homogeneous of degree $- \de_i$.

Suppose $M$ is a graded left $A$-module which is graded free with a
finite homogeneous base $\{b_1, \ldots , b_n\}$, where $\deg (b_i) =
\delta_i$. If we ignore the grading, it is well-known that
$\mathrm{End}_A(M) \cong M_n(A)$. When we take the grading into
account, we have that $\mathrm{End}_A(M) \conggr M_n(A)(d)$ for $(d)
= (\delta_1 , \ldots, \delta_n )$ (see
\cite[Prop.~2.10.5]{grrings}). Note that this isomorphism does not
depend on the order that the elements in the basis are listed. For
some permutation $\pi \in S_n$ we have that $\{b_{\pi (1)}, \ldots ,
b_{\pi (n)}\}$ is also a homogeneous base of $M$. So for $(d') =
(\delta_{\pi (1)} , \ldots, \delta_{\pi (n)} )$ we have $ M_n(A)(d')
\conggr \mathrm{End}_A(M) \conggr M_n(A)(d).$ Further for $(-d)= ( -
\de_1 , \ldots , - \de_n)$, the map $\varphi : A^n(-d) \ra M$
defined by $\varphi( e_i)= b_i$ is a graded $A$-module isomorphism,
and we write $M \conggr A^n(-d)$.

For a $\Ga$-graded ring $A$, and $(d)=(\de_1, \ldots ,\de_n) \in
\Ga^n$, $(a) = (\al_1 , \ldots \al_m) \in \Ga^m$, let
$$
M_{n \times m} (A)[d][a] =
\begin{pmatrix}
A_{-\de_1  + \al_1} & A_{ -\de_1 + \al_2} & \cdots &
A_{-\de_1  + \al_m} \\
A_{-\de_2  + \al_1} & A_{ -\de_2 + \al_2} & \cdots &
A_{-\de_2  + \al_m} \\
\vdots  & \vdots  & \ddots & \vdots  \\
A_{-\de_n  + \al_1} & A_{ -\de_n + \al_2} & \cdots & A_{ -\de_n +
\al_m}
\end{pmatrix}.
$$
So $M_{n\times m} (A)[d][a]$ consists of matrices with the
$ij$-entry in $A_{-\de_i +\al_j}$.

\begin{prop} \label{rndconggrrna}
Let $A$ be a $\Ga$-graded ring and let $(d)=(\de_1, \ldots
,\de_n)\in \Ga^n,$ $(a) = (\al_1 , \ldots , \al_m) \in \Ga^m$. Then
$A^n(d) \conggr A^m (a)$ as graded $A$-modules if and only if there
exists $$(r_{ij}) \in \GL_{n\times m} (A)[d][a].$$
\end{prop}

\begin{proof}
If $r = (r_{ij}) \in \GL_{n \times m} (A)[d][a]$, then there is a
graded $A$-module homomorphism
\begin{align*}
\mar_r :A^n (d)& \lra A^m(a) \\
(x_1 , \ldots, x_n) & \lmps (x_1 , \ldots, x_n) r.
\end{align*}
Since $r$ is invertible, there is a matrix $t \in \GL_{m \times n}
(A)$ with $rt=I_{n}$ and $tr = I_{m}$. So there is an $A$-module
homomorphism $\mar_t : A^m (a) \ra A^n(d)$, which is an inverse of
$\mar_r$. This proves that $\mar_r$ is bijective, and therefore it
is a graded $A$-module isomorphism.

Conversely, if $\phi : A^n (d) \conggr A^m(a)$, then we can
construct a matrix as follows. Let $e_i$ denote the basis element of
$A^n(d)$ with $1$ in the $i$-th entry and $0$ elsewhere. Then let
$\phi (e_i) = (r_{i1}, r_{i2} , \ldots , r_{im})$, and let $r=
(r_{ij})_{n \times m}$. It can be easily verified that $r \in M_{n
\times m} (A)[d][a]$. In the same way, using $\phi^{-1} : A^m (a)
\ra A^n(d)$ construct a matrix $t$. Let $e'_i$ denote the $i$-th
element of the standard basis for $A^m (a)$. Since $e_i = \phi^{-1}
\circ \phi (e_i) = r_{i1} \phi^{-1}(e'_1) + r_{i2} \phi^{-1} (e'_2)
\ldots + r_{im} \phi^{-1}(e'_m)$ for each $i$, and in a similar way
for $\phi \circ \phi^{-1}$, we can show that $rt=I_n$ and $tr= I_m$.
So $(r_{ij}) \in \GL_{n \times m} (A)[d][a]$.
\end{proof}

For convenience, in the above definition of $M_{n \times m} (A)
[d][a]$, if $(a) = (0, \ldots , 0)$, then we will write $M_{n \times
m} (A)[d]$ instead of $M_{n \times m} (A)[d][0]$. We let
$$
\Ga^*_{M_{n \times m}(A)} = \big \{ (d) \in \Ga^n : \GL_{n \times
m}(A)[d] \neq \emptyset \big \}.
$$
Then it is immediate from the above Proposition that $A^n(d) \conggr
A^n$ as graded $A$-modules if and only if $(d) \in \Ga^*_{M_n(A)}$.

A graded $A$-module $P= \bigoplus_{\ga \in \Ga} P_{\ga}$ is said to
be {\em graded projective} (resp.\ {\em graded faithfully
projective}) if $P$ is  projective (resp.\ faithfully projective)
as an $A$-module. Note that a graded $A$-module $P$ is projective as
an $A$-module if and only if $\Hom_{A \hygrmod}(P, -)$ is an exact
functor in $A \hygrmod$. We use $\Pgr (A)$ to denote the category of
graded finitely generated projective modules over $A$. The following
Proposition proves a partial result of Morita equivalence (only in
one direction), which we will use later in this section (after
Definition~\ref{gradeddfunctor}).

\begin{prop}[Morita Equivalence in the graded setting] \label{grmorita}
Let $A$ be a graded ring and let $(d) = (\de_1 , \ldots , \de_n)$,
where each $\de_i \in \Ga$. Then the functors
\begin{align*}
\psi : \Pgr( M_n(A)(d) ) & \lra \Pgr( A) \\
P & \longmapsto A^n(-d)\otimes_{M_n(A)(d)}P \\
\textrm{ and \;\;\;\; } \va : \Pgr( A) & \lra \Pgr ( M_n(A)(d)) \\
Q & \longmapsto A^n(d)\otimes_A Q
\end{align*}
form equivalences of categories.
\end{prop}

\begin{proof}
Observe that $A^n(d)$ is a graded $M_n(A)(d) \mhyphen A$-bimodule
and $A^n(-d)$ is a graded $A \mhyphen M_n(A)(d)$-bimodule. Then
\begin{align*}
\theta : \; A^n(-d)\otimes_{M_n(A)(d)} A^n(d)  & \lra A \\
(a_1 , \ldots , a_n ) \otimes (b_1, \ldots , b_n) & \longmapsto  a_1
b_1 + \cdots + a_n b_n ; \\
\textrm{ and \;\;\;\;}\sigma : \; A  & \lra
A^n(-d)\otimes_{M_n(A)(d)}
A^n(d)\\
a & \longmapsto (a , 0 , \ldots , 0) \otimes (1 , 0 , \ldots 0)
\end{align*}
are graded $A$-module homomorphisms with $\sigma \circ \theta = \id$
and $\theta \circ \sigma = \id$. Further
\begin{align*}
\theta' : \; A^n(d) \otimes_{A} A^n(-d) & \lra M_n(A)(d)\\
\begin{pmatrix}
a_1 \\ \vdots \\ a_n
\end{pmatrix}
\otimes
\begin{pmatrix}
b_1 \\ \vdots \\ b_n
\end{pmatrix}
& \longmapsto
\begin{pmatrix}
a_{1} b_1 & \cdots & a_{1} b_n  \\
\vdots  &   & \vdots  \\
a_{n} b_1  & \cdots & a_{n} b_n
\end{pmatrix}
\end{align*}
\begin{align*}
\textrm{and \;\;\;\; } \sigma' : \; M_n(A)(d) & \lra A^n(d) \otimes_{A} A^n(-d) \\
(m_{i,j}) & \longmapsto
\begin{pmatrix}
m_{1,1} \\ m_{2,1} \\ \vdots \\ m_{n,1}
\end{pmatrix}
\otimes
\begin{pmatrix}
1\\ 0 \\ \vdots \\0
\end{pmatrix}
+ \cdots +
\begin{pmatrix}
m_{1,n} \\ \vdots \\ m_{n-1,n} \\ m_{n,n}
\end{pmatrix}
\otimes
\begin{pmatrix}
0\\\vdots \\0 \\ 1
\end{pmatrix}
\end{align*}
are graded $M_n(A)(d)$-module homomorphisms with $\sigma' \circ
\theta' = \id$ and $\theta' \circ \sigma' = \id$. So
$A^n(-d)\otimes_{M_n(A)(d)} A^n(d) \conggr A$ and $A^n(d)
\otimes_{A} A^n(-d) \conggr M_n(A)(d)$ as $A \mhyphen A$-bimodules
and $M_n(A)(d) \mhyphen M_n(A)(d)$-bimodules respectively. Then for
$P \in \Pgr( M_n(A)(d) )$, $A^n(d) \otimes_{A} A^n(-d)
\otimes_{M_n(A)(d)} P \conggr P$ and for $Q \in\Pgr (A)$, $A^n(-d)
\otimes_{M_n(A)(d)} A^n(d) \otimes_A Q \conggr Q$, which shows that
$\psi$ and $\va$ are mutually inverse equivalences of categories.
\end{proof}

A graded $R$-algebra $A$ is called a {\em graded Azumaya algebra} if
$A$ is graded faithfully projective and $A \otimes_R A^{\op} \conggr
\End_R (A)$. We let $\text{Az}_\gr(R)$ denote the category of graded
Azumaya algebras over $R$ and $\mathcal Ab$ the category of abelian
groups. Note that a graded $R$-algebra which is an Azumaya algebra
(in the non-graded sense) is also a graded Azumaya algebra, since it
is faithfully projective as an $R$-module, and the natural
homomorphism $A \otimes_R A^{\op} \ra \End_R (A)$ is clearly graded.
So a graded central simple algebra over a graded field (as in
Theorem~\ref{gcsaazumayaalgebra}) is in fact a graded Azumaya
algebra.

\begin{defin} \label{gradeddfunctor}
An abstract functor $\mathcal F : \text{Az}_\gr(R) \ra \mathcal Ab$
is defined to be a {\bf graded \boldmath$\mathcal D$-functor} if it
satisfies the three properties below:
\begin{enumerate}
\item $\mathcal F(R)$ is the trivial group.

\item For any graded $R$-Azumaya algebra $A$ and for any $(d) = (\de_1
, \ldots , \de_k) \in \Ga^*_{M_k (R)}$, there is a homomorphism
$$\rho: \mathcal F (M_k (A) (d)) \lra \mathcal F (A)$$ such that the composition $$\mathcal
F (A) \lra \mathcal F (M_k (A) (d)) \lra \mathcal F (A)$$ is
$\eta_{k}$, where $\eta_k(x)=x^k$.

\item With $\rho$ as in property (2), $\ker(\rho)$ is $k$-torsion.
\end{enumerate}Note that these properties are well-defined since both $R$
and $M_k(A)(d)$ are graded Azumaya algebras over $R$.

\end{defin}


We set $K_i^{\gr} (R) = K_i (\Pgr (R))$, where $\mathcal P
\mathrm{gr}(R)$ is the category of graded finitely generated
projective $R$-modules and $K_i$ are the Quillen $K$-groups. Let $A$
be a graded ring with graded centre $R$. Then the graded $R$-linear
homomorphism $R \rightarrow A$ induces an exact functor $\mathcal P
\mathrm{gr}(R) \ra \Pgr (A)$, which, in turn, induces a group
homomorphism $K_i^{\gr} (R) \ra K_i^{\gr} (A)$. Then we have an
exact sequence
\begin{equation}\label{grexactseq}
1 \rightarrow \ZK[i]^{\gr}(A) \rightarrow K_i^{\gr}(R) \rightarrow
K_i^{\gr}(A) \rightarrow \CK[i]^{\gr}(A) \rightarrow 1
\end{equation}
where $\ZK[i]^{\gr}(A)$ and $\CK[i]^{\gr}(A)$ are the kernel and
cokernel of the map $K_i^{\gr}(R) \rightarrow K_i^{\gr}(A)$
respectively. Then $\CK[i]^\gr$ can be regarded as the following
functor
\begin{align*}
\CK[i]^{\gr}: \text{Az}_{\gr}(R) &\lra  \mathcal Ab \\
A &\longmapsto \CK[i]^{\gr}(A),
\end{align*}
and similarly for $\ZK[i]^{\gr}$. We will now show that
$\CK[i]^{\gr}$ is a graded $\mathcal D$-functor. Property (1) is
clear, since $R$ is commutative so $K_i^{\gr} (Z(R)) \ra
K_i^{\gr}(R)$ is the identity map. For property (2), let $\Pgr(A)$
and $\Pgr(M_k(A)(d))$ denote the categories of graded finitely
generated projective left modules over $A$ and $M_k(A)(d)$
respectively.

Then there are functors:
\begin{align}\label{grfor}
\phi : \Pgr(A) & \longrightarrow \Pgr (M_k(A)(d)) \\
 X  & \longmapsto M_k(A)(d) \otimes_A X \notag
\end{align}
and
\begin{align}\label{grfor1}
\psi : \Pgr (M_k(A)(d))   & \longrightarrow \Pgr (A) \\
 Y  & \longmapsto A^k(-d) \otimes_{M_k(A)(d)} Y. \notag
\end{align}
The functor $\phi$ induces a homomorphism from $K_i^{\gr}(A)$ to
$K_i^{\gr}(M_k(A)(d))$. By the graded version of the Morita Theorems
\label{moritapageref} (see Proposition~\ref{grmorita}), the functor
$\psi$ establishes a natural equivalence of categories, so it
induces an isomorphism from $K_i^{\gr}(M_k(A)(d))$ to
$K_i^{\gr}(A)$. For $X \in \Pgr(A)$, $\psi \circ \phi (X) \conggr
X^k(-d)$.  Since each $(d) \in \Ga^*_{M_k (R)}$, using a similar
argument to that of Proposition~\ref{rndconggrrna}, we have $X^k(-d)
\conggr X^k$. Since $K_i$ are functors which respect direct sums,
this induces a multiplication by $k$ on the level of $K$-groups.

Now the exact functors (\ref{grfor}) and (\ref{grfor1}) induce the
following commutative diagram:
\begin{equation*}
\begin{split}
\xymatrix{ K_i^{\gr}(R)  \ar[r] \ar[d]^{=} & K_i^{\gr}(A)  \ar[r]
\ar[d]^\phi
& \CK[i]^{\gr}(A) \ar[r] \ar[d]&1  \\
K_i^{\gr}(R)  \ar[r] \ar[d]^{\eta_k} & K_i^{\gr}(M_k(A)(d))  \ar[r]
\ar[d]^\psi_\cong & \CK[i]^{\gr}(M_k(A)(d)) \ar[r] \ar[d]^\rho& 1 \\
K_i^{\gr}(R)  \ar[r]  & K_i^{\gr}(A)  \ar[r] & \CK[i]^{\gr}(A)
\ar[r]& 1 }
\end{split}
\end{equation*}
where composition of the columns are $\eta_k$, proving property (2).
A diagram chase verifies that property (3) also holds. A similar
proof shows that $\ZK[i]^\gr$ is also a graded $\mathcal D$-functor.

\begin{thm} \label{grtorsionthm} Let $A$ be a graded Azumaya algebra
which is graded free over its centre $R$ of rank $n$, such that $A$
has a homogenous basis with degrees $(\de_1, \ldots , \de_n)$ in
$\Ga^*_{M_n(R)}$. Then $\mathcal F (A)$ is $n^2$-torsion, where
$\mathcal F$ is a graded $\mathcal D$-functor.
\end{thm}

\begin{proof}
Let $\{a_1 , \ldots , a_n\}$ be a homogeneous basis for $A$ over
$R$, and let $(d)= \big( \! \deg (a_1) , \ldots ,\deg( a_n) \big)
\in \Ga^*_{M_n (R)}$. Since $R$ is a graded Azumaya algebra over
itself, by (2) in the definition of a graded $\mathcal D$-functor,
there is a homomorphism $\rho:\mathcal F(M_n(R)(d)) \rightarrow
\mathcal F(R)$. But $\mathcal F(R)$ is trivial by property (1) and
therefore the kernel of $\rho$ is $\mathcal F(M_n(R)(d))$ which is,
by (3), $n$-torsion.

In the category $\text{Az}_\gr(R)$, the two graded $R$-algebra
homomorphisms $i:A \rightarrow A \otimes_R A^{\text{op}}$ and $r:
A^{\text{op}} \ra \End_R (A^{\op}) \ra M_n(R)(d)$ induce group
homomorphisms $\mathcal F(A) \ra \mathcal F(A \otimes_R A^{\op})$
and $\mathcal F(A \otimes_R A^{\op}) \ra \mathcal F (A \otimes_R
M_n(R)(d))$, where  $\mathcal F(A \otimes_R M_n(R)(d)) \cong
\mathcal F (M_n(A)(d))$. Further, the graded $R$-algebra isomorphism
$A \otimes_R A^{\text{op}} \conggr \End_R(A)$ from the definition of
a graded Azumaya algebra, combined with the graded isomorphism
$\End_R(A) \conggr M_n(R)(d)$, induces an isomorphism $\mathcal F(A
\otimes_R A^{\text{op}}) \cong \mathcal F(M_n(R)(d))$. Consider the
following diagram
\begin{equation*}
\begin{split}
\xymatrix{
& \mathcal F(A) \ar[d]^i \ar@/^/[ddr]^{\eta_n}&  \\
& \mathcal F(M_n(R)(d)) \cong  \mathcal F(A \otimes_R A^{\text{op}})
\ar[d]^r\;\;\;\;\;\;\;\;\;\;\;\;\;\;\;\;\;\;\;\;&  \\
& \mathcal F(M_n(A)(d))\ar[r]^-{\rho} & \mathcal F(A)}
\end{split}
\end{equation*}
which is commutative by property (2). It follows that $\mathcal
F(A)$ is $n^2$-torsion.
\end{proof}

\begin{thm}\label{grazumayafreethm}
Let $A$ be a graded Azumaya algebra which is graded free over its
centre $R$ of rank $n$, such that $A$ has a homogenous basis with
degrees $(\de_1, \ldots , \de_n)$ in $\Ga^*_{M_n(R)}$. Then for any
$i \geq 0$, $$K_i^{\gr}(A) \otimes \mathbb Z[1/n] \cong K_i^{\gr}(R)
\otimes \mathbb Z[1/n].$$
\end{thm}

\begin{proof} The argument before Theorem~\ref{grtorsionthm}
shows that $\CK[i]^{\gr}$ (and in the same manner $\ZK[i]^{\gr}$) is
a graded $\mathcal D$-functor, and thus by the theorem
$\CK[i]^{\gr}(A)$ and $\ZK[i]^{\gr}(A)$ are $n^2$-torsion abelian
groups. Tensoring the exact sequence (\ref{grexactseq}) by $\mathbb
Z[1/n]$, since $\CK[i]^{\gr}(A) \otimes \mathbb Z[1/n]$ and
$\ZK[i]^{\gr}(A) \otimes \mathbb Z[1/n]$ vanish, the result follows.
\end{proof}

\begin{cor}[\cite{hm}, Thm.~6]
Let $A$ be an Azumaya algebra free over its centre $R$ of rank $n$.
Then for any $i \geq 0$, $$K_i(A) \otimes \mathbb Z[1/n] \cong
K_i(R) \otimes \mathbb Z[1/n].$$
\end{cor}

\begin{proof}
By taking $\Ga$ to be the trivial group, this follows immediately
from Theorem~\ref{grazumayafreethm}.
\end{proof}

\begin{remark}
Note that a graded division algebra $A$ is strongly graded. By
Dade's Theorem \cite[Thm.~3.1.1]{grrings}, there is an additive
functor from the category of $A_0$-modules to the category of graded
$A$-modules which induces an equivalence of categories. This implies
that
\begin{equation} \label{dade}
K_i(A_0) \cong K_i^{\gr}(A).
\end{equation}
Let $D$ be a tame and Henselian valued division algebra with centre
$F$ of index $n$. Consider the associated graded division algebra
$\gr(D)$ with centre $\gr(F)$. We know $\gr(D)_0=\overline D$ and
$\gr(F)_0=\overline F$ and
$$[\gr(D):\gr(F)]=[\Gamma_D:\Gamma_F][\overline D:\overline F],$$
(see~\cite{hwcor}). If $D$ is unramified over $F$, i.e.,
$\Gamma_D=\Gamma_F$, then the assumption of
Theorem~\ref{grazumayafreethm} on the homogenous basis is satisfied,
so
$$
K_i^{\gr}(\gr(D)) \otimes \mathbb Z[1/n]  \cong K_i^{\gr}(\gr(F)) \otimes \mathbb Z[1/n].
$$
\end{remark}

We end the note with an example of a graded Azumaya algebra such
that its graded $K$-theory is not the same as the graded $K$-theory
of its centre.

\begin{example}\label{eggrktheory}
Consider the quaternion algebra $\mathbb H = \mathbb R \oplus
\mathbb R i \oplus \mathbb R j \oplus \mathbb R k$. Then $\mathbb H$
is an Azumaya algebra over $\mathbb R$ and it is a $\mathbb Z_2
\times \mathbb Z_2$-graded division ring. So $\mathbb H$ is in fact
a graded Azumaya algebra, which is strongly $\mathbb Z_2 \times
\mathbb Z_2$-graded. By Dade's Theorem, $K_0^\gr (\mathbb H) \cong
K_0 (\mathbb H_0) \cong  K_0 (\mathbb R) \cong \mathbb Z$. The
centre $Z(\mathbb H) = \mathbb R$ is a field and is trivially graded
by $\mathbb Z_2 \times \mathbb Z_2$. So $K_0^\gr
\big( Z( \mathbb H) \big) = K_0^\gr (\mathbb R) \cong \mathbb Z
\oplus \mathbb Z \oplus \mathbb Z\oplus \mathbb Z$.
 We remark that the graded Azumaya algebra $\mathbb H$
does not satisfy the conditions of Theorem~\ref{grazumayafreethm}.
\end{example}

\renewcommand{\refname}{Further reading}

\end{document}